\documentclass[a4paper,11pt,reqno]{amsart}
\usepackage{graphicx,amssymb}
\usepackage{enumitem, fullpage}
\usepackage{mathtools}

\usepackage{pstricks,tikz}
\usetikzlibrary{positioning,chains,fit,shapes,calc, fit}

\newtheorem{theorem}{Theorem}[section]
\newtheorem{lemma}[theorem]{Lemma}

\newtheorem{construction}[theorem]{Construction}

\numberwithin{equation}{section}

\def\eps{{\varepsilon}}
\def\Dr{\Delta_{r-1}}
\def\dr{\delta_{r-1}}
\def\ex{\text{ex}}
\def\pr{\pi_{r-1}}
\def\bF{\text{Fano}}

\title{Codegree Tur\'an density of complete $r$-uniform hypergraphs}

\thanks{
The first author is supported by EPSRC grant EP/P002420/1. 
The second author is partially supported by NSF grants DMS-1400073 and 1700622.}
\author{Allan Lo}
\address{School of Mathematics, University of Birmingham, Birmingham, B15 2TT, UK}
\email{s.a.lo@bham.ac.uk}
\author{Yi Zhao}
\address
{Department of Mathematics and Statistics, Georgia State University, Atlanta, GA 30303}
\email{yzhao6@gsu.edu}

\date{\today}

\begin{document}
\begin{abstract}
Let $r\ge 3$. Given an $r$-graph $H$, the minimum codegree $\dr(H)$ is the largest integer $t$ such that every $(r-1)$-subset of $V(H)$ is contained in at least $t$ edges of $H$. Given an $r$-graph $F$, the codegree  Tur\'an density $\gamma(F)$ is the smallest $\gamma >0$ such that every $r$-graph on $n$ vertices with $\dr(H)\ge (\gamma + o(1))n$ contains $F$ as a subhypergraph. Using results on the independence number of hypergraphs, we show that there are constants $c_1, c_2>0$ depending only on $r$ such that 
\[
1 - c_2 \frac{\ln t}{t^{r-1}} \le \gamma(K_t^r) \le 1 - c_1 \frac{\ln t}{t^{r-1}},
\]
where $K_t^r$ is the complete $r$-graph on $t$ vertices.
This gives the best general bounds for $\gamma(K_t^r)$. 
\end{abstract}

\maketitle

\section{introduction}
An $r$-uniform hypergraph ($r$-graph) $H$ consists of a vertex set $V(H)$ and an edge set $E(H)$, which is a family of $r$-subsets of $V(H)$. A fundamental problem in extremal combinatorics is to determine the Tur\'an number $\ex(n, F)$, which is the largest number of edges in an $r$-graph on $n$ vertices not containing a given $r$-graph $F$ as a subhypergraph (namely, $F$-free). When $r\ge3$, we only know $\ex(n, F)$, or its asymptotics $\pi(F):= \lim_{n\to \infty} \ex(n, F)/\binom nr$ for very few $F$. Let $K_t^r$ denote the complete $r$-graph on $t$ vertices. Determining $\pi(K_t^r)$ for any $t>r\ge3$ is a well known open problem, in particular, Tur\'an~\cite{MR0018405} conjectured in 1941 that $\pi(K_4^3)= 5/9$. The best (general) bounds for $\pi(K_t^r)$ are due to Sidorenko~\cite{MR635252} and de Caen~\cite{MR734038}
\begin{equation}\label{eq:1}
   1- \left(\frac{r-1}{t-1}\right)^{r-1}\leq \pi(K^r_{t})\leq 1-\frac{1}{{t-1\choose r-1}}.
\end{equation}
For more Tur\'an-type results on hypergraphs, see surveys \cite{MR1161467,MR2866732}.

A natural variation on the Tur\'an problem is to ask how large the minimum $\ell$-degree can be in an $F$-free $r$-graph. Given an $r$-graph $H$, the \emph{degree} $\deg(S)$ of a set $S\subset V(H)$ is the number of the edges that contain $S$. Given $1\le \ell < r$, the minimum $\ell$-degree $\delta_{\ell}(H)$ is the minimum $\deg(S)$ over all $S\subset V(H)$ of size $\ell$. Mubayi and Zhao \cite{MR2337241} introduced the \emph{codegree Tur\'an number} $\ex_{r-1}(n, F)$, which is the largest $\dr(H)$ among all $F$-free $r$-graphs on $n$ vertices, and \emph{codegree (Tur\'an) density} $\pi_{r-1}(F) := \lim_{n\to \infty} \ex_{r-1}(n, F)/n$ (it was shown \cite{MR2337241} that this limit exists). The corresponding $\ell$-degree Tur\'an number $\ex_{\ell}(n, F)$ and density $\pi_{\ell}(F)$ were defined similarly and studied by Lo and Markstr\"om \cite{MR3248027}.\footnote{A simple averaging argument shows that $\pi_1(F) = \pi(F)$ for every $F$.}

Most codegree Tur\'an problems do not seem easier than the original Tur\'an problems. We only know the codegree densities of the following $r$-graphs. Let $\bF$ denote the Fano plane (a 3-graph on seven vertices and seven edges). Mubayi \cite{MR2171370} showed that $\pi_2(\bF)= 1/2$ and Keevash \cite{MR2478542} later showed that $\ex_2(n, \bF)= \lfloor n/2 \rfloor$ for sufficiently large $n$ (DeBiasio and Jiang~\cite{MR3131883} gave another proof). Keevash and Zhao \cite{MR2354709} studied the codegree density for other projective geometries and constructed a family of 3-graphs whose codegree densities are $1- 1/t$ for all integers $t\ge 1$. Falgas-Ravry, Marchant, Pikhurko, and Vaughan \cite{MR3384831} determined $\ex_2(F_{3,2})$ for sufficiently large $n$, where $F_{3, 2}$ is the 3-graph on $\{1, 2, 3, 4, 5\}$ with edges $123, 124, 125, 345$. Falgas-Ravry, Pikhurko, Vaughan and Volec~\cite{FPVV} also proved that $\pi_2(K_4^{3-})= 1/4$, where $K_4^{3-}$ is the (unique) 3-graphs on four vertices with three edges.  

In this note we obtain asymptotically matching bounds for $\pr(K_t^r)$ for any fixed $r\ge 3$ and sufficiently large $t$. Since its value is close to one, it is more convenient to write $\pr(K_t^r)$ in the complementary form. Given an $r$-graph $H$ and $\ell<r$, let $\Delta_{\ell}(H)$ denote the \emph{maximum $\ell$-degree} of $H$ and $\alpha(H)$ denote the \emph{independence number} (the largest size of a set of vertices containing no edge) of $H$.  Define 
\[
T_{\ell} (n,t,r) = \min \left\{ \Delta_{\ell}(H): H \text{ is an $r$-graph on $n$ vertices with } \alpha (H) <t \right\}
\]
and $\tau_{\ell}(t,r) = \lim_{n \rightarrow \infty} T_{\ell}(n,t,r) / \binom{n- \ell }{r- \ell}$. 
It is clear that $T_{\ell}(n, t, r) = \binom{n-\ell}{r - \ell} - \ex_{\ell}(n, K_t^r)$ and $\tau_{\ell}(t,r) =1 - \pi_{\ell}(K_t^r)$.
Falgas-Ravry~\cite{MR3158267} showed that $\tau_2(t, 3) \le 1/(t-2)$ for $t\ge 4$ while Lo and Markstr\"om~\cite{MR3248027} showed that $\tau_{r-1}(t, r) \le 1/(t-r+1)$ for $t>r\ge 3$. Recently Sidorenko~\cite{Sid17} used zero-sum-free sequences in $\mathbb{Z}_3^d$ to get $\tau_2(t, 3) \le O(\frac{1}{t\ln t})$.

We show that $\tau_{r-1}(t, r) = \Theta(\ln t/t^{r-1})$ as $t\to \infty$.
\begin{theorem}\label{thm:main}
For all $r\ge 3$, there exist $c_1, c_2>0$ such that 
\[
c_1  \ln t /  t^{r-1} \le \tau_{r-1} (t,r)  \le c_2 \ln t /  t^{r-1}.
\]
\end{theorem}

In fact, the upper bound immediately follows from a construction of Kostochka, Mubayi and Verstra\"ete~\cite{MR3158630} (see Construction~\ref{con:KMV1}).
The lower bound can be deduced from either the main result of \cite{MR3158630} or a result of Duke, Lefmann, and R\"odl \cite{MR1370956}. However, since both results require $\Delta_{r-1}(H)= o(n)$, we need to extend them slightly by allowing $\Delta_{r-1}(H)$ to be a linear function of $n$ (see Theorem~\ref{thm:main2}).

We prove Theorem~\ref{thm:main} in the next section and give concluding remarks and open questions in the last section. 

\section{Proof of Theorem~\ref{thm:main}}
A \emph{partial Steiner $(n, r, \ell)$-system} is an $r$-graph on $n$ vertices in which every set of $\ell$ vertices is contained in at most one edge. 
R\"odl and \u{S}inajov\'a~\cite{MR1248185} showed that there exists $a_2>0$ such that for every $m$, there is a partial Steiner $(m, r, r-1)$-system $S$ with $\alpha (S) \le a_2 (m \ln m)^{1/(r-1)}$.
Kostochka, Mubayi and Verstra\"ete~\cite[Section 3.1]{MR3158630} used the blowup of this Steiner system to obtain the following construction. A similar construction (but not using the result of \cite{MR1248185}) was given in \cite{MR3158267}.

\begin{construction}\cite{MR3158630} \label{con:KMV1}
Let $S$ be the partial Steiner $(m, r, r-1)$-system given by R\"odl and \u{S}inajov\'a.
Let $V$ be a union of disjoint sets $V_1, \dots, V_m$ each of size~$d$.
For each edge $e= \{i_1, \dots, i_r\}$ of~$S$, let $E_e : = \{ v_{1} v_2 \dots v_r \colon v_j \in V_{i_j}$ for $j \in [r] \}$.
Let $H$ be the $r$-graph with vertex set~$V$ and edge set $\bigcup_{i \in [m]} \binom{V_i}{r} \cup \bigcup_{ e \in S}  E_e$. It is easy to see that 
\begin{align*}
\Delta_{r-1} (H) = d \quad \text{and} \quad \alpha(H) = (r-1) \alpha (S)  \le a_2 (r-1) (m \ln m)^{\frac1{r-1}}. 
\end{align*}
\end{construction}

Construction~\ref{con:KMV1} will be used to prove the upper bound of Theorem~\ref{thm:main}. 
The lower bound of~Theorem~\ref{thm:main} follows from the following theorem, which will be proved at the end of the section.

\begin{theorem} \label{thm:main2}
For all $r\ge 3$, there exist $c_{0}, \delta_0>0$ such that for every $0< \delta \le \delta_0$, the following holds for sufficiently large $n$. 
Every $r$-graph on $n$ vertices with $\Delta_{r-1}(H) \le \delta n $ satisfies
$ \alpha (H) \ge  c_{0} \left( \frac{1}{\delta} \ln \frac{1}{\delta} \right)^{1/(r-1)}$.
\end{theorem}

\smallskip
\begin{proof}[Proof of Theorem~\ref{thm:main}]
Fix $r\ge 3$. Without loss of generality, we assume that $t$ is sufficiently large.
We first prove the upper bound with $c_2 = (r-1)^r a_2^{r-1}$, where $a_2$ is from Construction~\ref{con:KMV1}. 
Our goal is to construct $r$-graphs $H$ on $n$ vertices (for infinitely many $n$) with $\alpha(H)< t$ and $\Delta_{r-1}(H) \le c_2 n\ln t / t^{r-1}$. 
To achieve this, we apply Construction~\ref{con:KMV1} with $m =  \lceil t^{r-1}/ (c_2 \ln t) \rceil$ and $d= n/m \le c_2 n\ln t / t^{r-1}$ obtaining an $r$-graph $H$ on $n$ vertices with $\Dr(H)= d$ and $\alpha(H)\le a_2 (r-1) (m \ln m)^{1/(r-1)}$. 
Since $t$ is sufficiently large, it follows that $\ln \left\lceil\frac{t^{r-1}}{c_2 \ln t}\right\rceil < \ln t^{r-1} - 1$ and
\[
m \ln m = \left\lceil \frac{t^{r-1}}{c_2 \ln t}\right\rceil \ln \left\lceil\frac{t^{r-1}}{c_2 \ln t}\right\rceil < \left( \frac{t^{r-1}}{c_2 \ln t} + 1 \right) \left( \ln t^{r-1} - 1 \right) <  \frac{(r-1) t^{r-1}}{c_2}.
\]
Consequently $\alpha(H)< a_2 (r-1) ( \frac{(r-1) t^{r-1}}{c_2} )^{1/(r-1)} = t$ by the choice of $c_2$.

\smallskip
We now prove the lower bound. Suppose $c_{0}, \delta_0$ are as in Theorem~\ref{thm:main2}.
Let $c_1 = (r-1) c_0^{r-1}/2$ and $\delta = c_1 \ln t / t^{r-1}$. Since $t$ is large, we have $\delta \le \delta_0$.
Let $n$ be sufficiently large.
We need to show that every $r$-graph $H$ on $n$ vertices with $\alpha(H)< t$ satisfies $\Dr(H) \ge \delta n$.
Indeed, by Theorem~\ref{thm:main2}, any $r$-graph $H$ on $n$ vertices with $\Dr(H) = d\le \delta n$ satisfies
\begin{align*}
	\alpha(H) & 
	\ge c_0 \left( \frac{1}{\delta} \ln \frac{1}\delta \right)^{1/(r-1)} 
	 > c_0 \left( \frac{ t^{r-1}}{2c_1 \ln t} \ln t^{r-1} \right)^{1/(r-1)}
	= t
\end{align*}
because $t$ is large and and $c_1 = (r-1) c_0^{r-1}/2$.
\end{proof}

\smallskip

The rest of the section is devoted to the proof of Theorem~\ref{thm:main2}. We need \cite[Theorem 1]{MR3158630} of Kostochka, Mubayi, Verstra\"ete and \cite[Lemma~2.1]{MR2337241} of Mubayi and Zhao.\footnote{Alternatively we could apply \cite[Theorem 3]{MR1370956} of Duke, Lefmann, and R\"odl -- we choose \cite[Theorem 1]{MR3158630} because it provides a better constant.}

\begin{theorem}\cite{MR3158630} \label{thm:KMV}
For all $r \ge 3$, there exists $b_1 >0$ such that every $r$-graph with $\Delta_{r-1}(H) \le d$ for some $0< d < n/ (\ln n)^{3 (r-1)^2}$ satisfies
$ \alpha (H) \ge b_1 \left( \frac{n}{d} \ln \frac{n}{d} \right)^{1/(r-1)}$.
\end{theorem}

\begin{lemma}\cite{MR2337241} \label{lma:MubayiZhao}
Let $r \ge 2$ and $\eps >0$.
Let $m$ be the positive integer such that $m \ge 2 ( r - 1 ) / \eps$ and $\binom{m}{r-1} e^{ - \eps^2 (m-r+1)/12} \le 1/2$.
Every $r$-graph $H$ on $n \ge m$ vertices contains an induced sub-$r$-graph~$H'$ on~$m$ vertices with $\Delta_{r-1}(H') / m \le \Delta_{r-1} (H)/n + \eps$.
\end{lemma}

\begin{proof}[Proof of Theorem~\ref{thm:main2}]
Fix $r \ge 3$. Let $0<\delta_0<1/4$ such that
\begin{align} \label{eq:delta}
24 (r-1) \ln \left( \left\lceil \frac1{\delta^4} \right\rceil \right) \le \frac1{\delta^2} \quad \text{and} \quad \frac{1}{\delta^4}\le \exp \left( \left(\frac1{2\delta}\right)^{\frac1{3(r-1)^2}} \right) - 1
\end{align}
for all $0< \delta\le \delta_0$.
Let $m= \lceil 1/\delta^4 \rceil$. We claim that $m$ satisfies the assumption of Lemma~\ref{lma:MubayiZhao} when $\eps= \delta$. Indeed,
it follows from the first inequality of \eqref{eq:delta} that 
\begin{align*} 
m  \ge \frac{24 (r-1) \ln m}{\delta^2} > \frac{2(r-1)}{\delta}, 
\end{align*}
which further implies that
\begin{align*}
	\binom{m}{r-1} e^{ - \delta^2 (m-r+1)/12} & \le \frac12 m^{r-1} e^{- \frac{m \delta^2}{24 } } \le \frac12.
\end{align*}

Let $c_0 =  4^{-1/(r-1)} b_1$, where $b_1$ is defined in Theorem~\ref{thm:KMV}.
Suppose $H$ is an $r$-graph on $n \ge m$ vertices with $\Delta_{r-1}(H) \le \delta n$.
By Lemma~\ref{lma:MubayiZhao}, there exists an induced subhypergarph $H'$ on $m$ vertices such that  
\begin{align*}
\Delta_{r-1}(H')  \le 2 \delta m  <  \frac{m}{ (\ln m)^{3 (r-1)^2}},
\end{align*}
which follows from the second inequality of \eqref{eq:delta} and $m= \lceil 1/\delta^4 \rceil$.
We now apply Theorem~\ref{thm:KMV} to $H'$ with $d= 2\delta m$ and obtain that 
\begin{align*}
 \alpha(H) & \ge \alpha (H') 
 \ge b_1\left(  \frac{1}{2 \delta } \ln \frac{1}{2 \delta }  \right)^{1/(r-1)} 
	\ge b_1\left(  \frac{1}{4 \delta } \ln \frac{1}{ \delta }  \right)^{1/(r-1)} 
	= c_0 \left(\frac{1}{\delta} \ln \frac{1}{\delta} \right)^{1/(r-1)}
\end{align*}
by the choice of $c_0$ and the assumption that $\delta \le 1/4$.
\end{proof}

\section{Concluding remarks}

Theorem \ref{thm:main} shows that $c_1 \ln t/ t^{r-1} \le \tau_{r-1} (t,r)  \le c_2 \ln t / t^{r-1}$.
Our proofs of Theorems~\ref{thm:main} and \ref{thm:main2} together give that $c_1 =   (r-1)  b_1^{r-1}/8$, where $b_1$ comes from Theorem~\ref{thm:KMV}.
A slightly more careful calculation allows us to take $c_1 = ( 1+o_t(1) ) (r-1)  b_1^{r-1}$ (where $o_t(1)\to 0$ as $t \rightarrow \infty$). 
The equation (7) in \cite{MR3158630} shows that $b_1^{r-1} = (1 + o_r(1)) (r-3)!/3$ and thus 
\[
c_1 = (1+ o_r(1))\frac{r-1}3 (r-3)!.
\]
On the other hand, our proof of Theorem~\ref{thm:main} gives $c_2 = (r-1)^r a_2^{r-1}$, where $a_2$ comes from Construction~\ref{con:KMV1}.
Unfortunately, we do not know the smallest $a_2$ such that there is a partial Steiner $(m, r, r-1)$-system $S$ with $\alpha (S) \le a_2 (m \ln m)^{1/(r-1)}$ for every $m$. However, the random construction in \cite[Section 3.2]{MR3158630} yields a constant that asymptotically equals $b_1$ but requires $\ln \Delta_{r-1}(H) = o(\ln n)$. Nevertheless, we can use the blowup of this construction and add some additional edges when $r\ge 4$ to derive that\footnote{
For example, when $r$ is even, we add all the $r$-sets that lie inside one vertex class and the $r$-sets that intersect $r/2$ vertex classes each with exactly two vertices.} 

\begin{align*}
c_2 = 
\begin{cases}
	( 1+o_t(1) ) r \cdot r!  & \text{if $r$ is even},\\
	( 1+o_t(1) ) 2^{r-1} r \cdot r! & \text{if $r$ is odd.}
\end{cases}
\end{align*}
When $r$ is even, above refined values of $c_1$ and $c_2$ differ by a factor of~$3r^3$ asymptotically.
We tend to believe that  
$\tau_{r-1}(t,r) \sim r \cdot r! \ln t / t^{r-1}$ when $t\gg r\gg 1$.

Given any $r$-graph $H$ on $n$ vertices, $\Delta_{\ell}(H)/ \binom{n- \ell}{r- \ell}$ is an increasing function of $\ell$. As a result, $\tau_{\ell} (t,r)$ is an increasing function of $\ell$.  When $t\to \infty$, 
we have $\tau_1 (t,r) = 1 - \pi(K_t^r) = \Theta ( 1/t^{r-1} )$ from~\eqref{eq:1} and $\tau_{r-1} (t,r) = \Theta(\ln t / t^{r-1})$
from Theorem~\ref{thm:main}. 
Putting these together, we have
\[
\Theta \left(\frac1{t^{r-1}} \right) = \tau_{1} (t,r)  \le \tau_2 (t,r) \le \dots \le \tau_{r-1} (t,r) = \Theta \left( \frac{\ln t}{t^{r-1}} \right).
\]
It is interesting to know if $\tau_{\ell} (t,r) = \Theta (\ln t / t^{r-1})$ for all $\ell\ge 2$. 

\section*{Acknowledgment}

The authors would like to thank Sidorenko and two referees for their valuable comments.



\begin{thebibliography}{10}

\bibitem{MR734038}
D.~de~Caen.
\newblock Extension of a theorem of {M}oon and {M}oser on complete subgraphs.
\newblock {\em Ars Combin.}, 16:5--10, 1983.

\bibitem{MR3131883}
L.~DeBiasio and T.~Jiang.
\newblock On the co-degree threshold for the {F}ano plane.
\newblock {\em European J. Combin.}, 36:151--158, 2014.

\bibitem{MR1370956}
R.~A. Duke, H.~Lefmann, and V.~R\"odl.
\newblock On uncrowded hypergraphs.
\newblock In {\em Proceedings of the {S}ixth {I}nternational {S}eminar on
  {R}andom {G}raphs and {P}robabilistic {M}ethods in {C}ombinatorics and
  {C}omputer {S}cience, ``{R}andom {G}raphs '93'' ({P}ozna\'n, 1993)},
  volume~6, pages 209--212, 1995.

\bibitem{MR3158267}
V.~Falgas-Ravry.
\newblock On the codegree density of complete 3-graphs and related problems.
\newblock {\em Electron. J. Combin.}, 20(4):Paper 28, 14, 2013.

\bibitem{MR3384831}
V.~Falgas-Ravry, E.~Marchant, O.~Pikhurko, and E.~R. Vaughan.
\newblock The codegree threshold for 3-graphs with independent neighborhoods.
\newblock {\em SIAM J. Discrete Math.}, 29(3):1504--1539, 2015.

\bibitem{FPVV}
V.~Falgas-Ravry, O.~Pikhurko, E.~R. Vaughan, and J.~Volec.
\newblock {The codegree threshold of ${K}_4^-$}.
\newblock {\em Electron. Notes Discrete Math}, 61:407--413, 2017.

\bibitem{MR1161467}
Z.~F\"uredi.
\newblock Tur\'an type problems.
\newblock In {\em Surveys in combinatorics, 1991 ({G}uildford, 1991)}, volume
  166 of {\em London Math. Soc. Lecture Note Ser.}, pages 253--300. Cambridge
  Univ. Press, Cambridge, 1991.

\bibitem{MR2478542}
P.~Keevash.
\newblock A hypergraph regularity method for generalized {T}ur\'an problems.
\newblock {\em Random Structures Algorithms}, 34(1):123--164, 2009.

\bibitem{MR2866732}
P.~Keevash.
\newblock Hypergraph {T}ur\'an problems.
\newblock In {\em Surveys in combinatorics 2011}, volume 392 of {\em London
  Math. Soc. Lecture Note Ser.}, pages 83--139. Cambridge Univ. Press,
  Cambridge, 2011.

\bibitem{MR2354709}
P.~Keevash and Y.~Zhao.
\newblock Codegree problems for projective geometries.
\newblock {\em J. Combin. Theory Ser. B}, 97(6):919--928, 2007.

\bibitem{MR3158630}
A.~Kostochka, D.~Mubayi, and J.~Verstra\"ete.
\newblock On independent sets in hypergraphs.
\newblock {\em Random Structures Algorithms}, 44(2):224--239, 2014.

\bibitem{MR3248027}
A.~Lo and K.~Markstr\"om.
\newblock {$\ell$}-degree {T}ur\'an density.
\newblock {\em SIAM J. Discrete Math.}, 28(3):1214--1225, 2014.

\bibitem{MR2171370}
D.~Mubayi.
\newblock The co-degree density of the {F}ano plane.
\newblock {\em J. Combin. Theory Ser. B}, 95(2):333--337, 2005.

\bibitem{MR2337241}
D.~Mubayi and Y.~Zhao.
\newblock Co-degree density of hypergraphs.
\newblock {\em J. Combin. Theory Ser. A}, 114(6):1118--1132, 2007.

\bibitem{MR1248185}
V.~R\"odl and E.~\v{S}i\v{n}ajov\'a.
\newblock Note on independent sets in {S}teiner systems.
\newblock In {\em Proceedings of the {F}ifth {I}nternational {S}eminar on
  {R}andom {G}raphs and {P}robabilistic {M}ethods in {C}ombinatorics and
  {C}omputer {S}cience ({P}ozna\'n, 1991)}, volume~5, pages 183--190, 1994.

\bibitem{Sid17}
A.~{Sidorenko}.
\newblock {Extremal problems on the hypercube and the codegree
  Tur\'an density of complete $r$-graphs}.
\newblock {\em ArXiv e-prints}, October 2017.

\bibitem{MR635252}
A.~F. Sidorenko.
\newblock Systems of sets that have the {$T$}-property.
\newblock {\em Vestnik Moskov. Univ. Ser. I Mat. Mekh.}, (5):19--22, 1981.

\bibitem{MR0018405}
P.~Tur\'an.
\newblock Eine {E}xtremalaufgabe aus der {G}raphentheorie.
\newblock {\em Mat. Fiz. Lapok}, 48:436--452, 1941.

\end{thebibliography}

\end{document}